\documentclass{amsart}
\usepackage{amsmath}
\usepackage{mathrsfs}
\usepackage{amsfonts}
\usepackage{amssymb,color}
\usepackage{graphicx}
\usepackage[square,numbers]{natbib}
\usepackage{verbatim}
\usepackage{enumerate}
\usepackage{float}
\usepackage{hyperref}
\usepackage{comment}
\usepackage{tikz}
\newtheorem{theorem}{Theorem}

\newtheorem{lemma}[theorem]{Lemma}
\newtheorem{proposition}[theorem]{Proposition}

\newtheorem*{theorem*}{Theorem}
\newtheorem*{lemma*}{Lemma}

\theoremstyle{definition}
\newtheorem{definition}[theorem]{Definition}

\theoremstyle{remark}
\newtheorem{remark}[theorem]{Remark}

\numberwithin{theorem}{section}
\numberwithin{equation}{section}

\def\XXint#1#2#3{{\setbox0=\hbox{$#1{#2#3}{\int}$}
		\vcenter{\hbox{$#2#3$}}\kern-.5\wd0}}

\newcommand{\bbR}{\mathbb{R}}
\newcommand{\bbN}{\mathbb{N}}

\newcommand{\Lip}{{\rm Lip}}
\newcommand{\diam}{{\rm diam}}
\newcommand{\R}{{\mathbb R}}
\newcommand{\Z}{{\mathbb Z}}
\newcommand{\N}{{\mathbb N}}

\newcommand{\calH}{\mathcal{H}}

\begin{document}
	\title[UDS in Laakso Space]{Universal Differentiability Sets in Laakso Space}

	\author[Sylvester Eriksson-Bique]{Sylvester Eriksson-Bique}
	\address[Sylvester Eriksson-Bique]{Department of Mathematics and Statistics, University of Jyvaskyla, Seminaarinkatu 15, PO Box 35, FI-40014 University of Jyvaskyla, Finland}
	\email[Sylvester Eriksson-Bique]{sylvester.d.eriksson-bique@jyu.fi}
	
	\author[Andrea Pinamonti]{Andrea Pinamonti}
	\address[Andrea Pinamonti]{Department of Mathematics, University of Trento, Via Sommarive 14, 38123 Povo (Trento), Italy}
	\email[Andrea Pinamonti]{Andrea.Pinamonti@unitn.it}
	
	\author[Gareth Speight]{Gareth Speight}
	\address[Gareth Speight]{Department of Mathematical Sciences, University of Cincinnati, 2815 Commons Way, Cincinnati, OH 45221, United States}
	\email[Gareth Speight]{Gareth.Speight@uc.edu}
	
	
	\begin{abstract}
We show that there exists a family of mutually singular doubling measures on Laakso space with respect to which real-valued Lipschitz functions are almost everywhere differentiable. This implies that there exists a measure zero universal differentiability set in Laakso space. Additionally, we show that each of the measures constructed supports a Poincar\'e inequality.
	\end{abstract}
	
	\maketitle

\section{Introduction}

Rademacher's theorem states that Lipschitz functions between Euclidean spaces are differentiable Lebesgue almost everywhere. This result has many extensions and applications. One direction of research extends Rademacher's theorem to more general spaces, such as Banach spaces \cite{BL00}, Carnot groups \cite{Pan89}, and metric measure spaces \cite{Che99}. Another direction of research asks to what extent Rademacher's theorem is optimal \cite{ACP10, CJ15, PS15, Pre90, Zah46}. The present paper contributes to this study in the setting of Laakso space, a metric measure space in which a version of Rademacher's theorem holds. Laakso space (Definition \ref{def_laakso}) is of the form $F:=(I\times K)/\sim$, where $K$ is the middle third Cantor set and $\sim$ is a suitable equivalence relation. It was first constructed by Laakso \cite{Laa00} who gave a family of spaces, depending on parameters, to show there exists an Ahlfors $Q$-regular metric measure space of any dimension $Q>1$ which supports a Poincar\'e inequality. See also \cite{Cap} for a nice overview of the main properties of Laakso space.

It has been known for some time that Rademacher's theorem does not admit a converse for Lipschitz maps $f\colon \mathbb{R}^{n}\to \mathbb{R}$ for $n\geq 2$, see e.g. \cite{DM11, DM14}. To be more precise, there exists a Lebesgue measure zero set $N\subset \mathbb{R}^{n}$ containing a point of differentiability for every real-valued Lipschitz map on $\mathbb{R}^{n}$. Such a set $N$ is often called a universal differentiability set (UDS). In addition to the Euclidean case, measure zero UDS are known to exist in some classes of Carnot groups, which include all step two Carnot groups and examples of arbitrarily high step \cite{PS16, LPS17, PS18}. The key technique used to prove these results is a refinement of the fact that existence of a maximal directional derivative implies differentiability \cite{Fit84}. Surprisingly, the second and third authors showed in \cite{CPS22} that this fact does not have a simple analogue in Laakso space. This left open the question of whether measure zero UDS exist in Laakso space. The present paper answers this question by showing that they do, by means of a different method.

Another natural way to study optimality of Rademacher's theorem is to ask whether Lipschitz functions can be differentiable almost everywhere with respect to another measure on the same space. It is known that in Euclidean spaces such a measure must be absolutely continuous with respect to Lebesgue measure \cite{DR16} and in Carnot groups it must be absolutely continuous with respect to the natural Haar measure \cite{DMMPR}. In this paper we show that these results do not extend to the Laakso space with its natural Hausdorff measure. In particular, there exists a family of mutually singular measures with respect to which real-valued Lipschitz functions are differentiable almost everywhere. It should be stressed that the idea behind the construction of the measures comes from work of Schioppa \cite{Sch15}, who constructed a family of mutually singular measures on a  metric measure space which are all doubling and support a Poincar\'e inequality. The space Schioppa used is different from that of Laakso.

We now describe our main results. Our first main result gives a family of mutually singular doubling measures for which Rademacher's theorem holds.

\begin{theorem}\label{rademachermuw}
There exist doubling measures $\mu_{w}$ on $F$ for each $w\in (0,1)$ so that 
\begin{enumerate}
    \item $\mu_{w}$ and $\mu_{w'}$ are mutually singular whenever $w\neq w'$, and
    \item for each $w\in (0,1)$, every Lipschitz map $f\colon F\to \mathbb{R}$ is differentiable almost everywhere with respect to $\mu_{w}$.
\end{enumerate}
\end{theorem}

 Each measure $\mu_{w}$ is the push forward of $\mathcal{H}^{1}\times \nu_{w}$ under the quotient map, where $\nu_{w}$ is a suitable measure on $K$. The measure $\nu_{w}$ is defined by assigning a proportion $w$ of the measure to the left similar copy of $K$ and a proportion $1-w$ of the measure to the right similar copy of $K$ at any stage in the construction of $K$. These measures are mutually singular for distinct $w$, while the measure $\nu_{1/2}$ is proportional to the natural Hausdorff measure on $K$. Due to the structure of $\mu_{w}$ (not least the Euclidean behavior in the $I$ direction and the fact that $\mu_{w}$ is doubling so the Lebesgue density theorem holds), one can adapt the explicit proof of Rademacher's theorem given in \cite{CPS22} to $\mu_{w}$ for any $w\in (0,1)$. This gives Theorem \ref{rademachermuw}.

Our second main result deduces the existence of measure zero UDS in Laakso space. It is an immediate consequence of Theorem \ref{rademachermuw}. To describe the UDS, first denote the left and right similar copies of the middle third Cantor set $K$ by $K_{0}$ and $K_{1}$ respectively. Similarly we define $K_{a}$ for any finite string $a$ of $0$'s and $1$'s. For integer $n\geq 1$, define $X_{n}:K\to \mathbb{R}$ by $X_{n}= 1$ on $K_{a0}$ and $X_{n}= 0$ on $K_{a1}$ for any binary string $a$ of length greater or equal to $0$. We set $Q:=1+ (\log 2/\log 3)$, noting that $F$ is Ahlfors $Q$-regular.

 \begin{theorem}\label{thmuds}
There exists a Borel set $N\subset F$ with $\mathcal{H}^{Q}(N)=0$ such that every Lipschitz map $f\colon F\to \mathbb{R}$ is differentiable at a point of $N$.

More precisely, we can choose $N$ to be $q(I\times E_{w})$ for any $w\neq 1/2$, where the map $q\colon I\times K\to F$ is the quotient mapping sending $x$ to $[x]$ and
\[E_{w}=\{x\in K: \frac{1}{n}\sum_{k=1}^{n}X_{n}(x)\to w \mbox{ as }n\to \infty\}.\]
 \end{theorem}

Historically, Rademacher's theorem for general metric spaces was discovered by Cheeger for spaces that satisfy a Poincar\'e inequality \cite{Che99}.
Despite this, a special feature of our work is that we do not use a Poincar\'e inequality to prove the above results. However, as recognized in \cite{SEB}, for a special class of metric measure spaces (called RNP-differentiability spaces), Rademacher theorem implies a Poincar\'e inequality. Thus, it is relevant to study whether the presently studied spaces also satisfy a Poincar\'e inequality. This also draws a closer parallel to the work of Schioppa in \cite{Sch15}, and shows that this work completely extends to Laakso spaces. Indeed, it shows that the Laakso space admits an uncountable family of mutually singular measures supporting a Poincar\'e inequality. We defer to Section \ref{sec:pi} for definitions and a more detailed discussion.

\begin{theorem}\label{thm:PI}
  For every $w\in(0,1)$ the space $(F,d,\mu_w)$ satisfies a $(1,1)$-Poincar\'e inequality.  
\end{theorem}

As pointed out by Laakso \cite{Laa00}, similar spaces can be constructed if the two-part Cantor set is replaced by a three-part Cantor set or another self similar set composed of parts which have positive distance from each other and where every part is isometric to the whole after scaling by a multiplicative factor. The authors would expect (but have not checked in details) that similar methods would yield measure zero differentiability sets in at least some of these examples. This is because the properties of the measures needed in our work are relatively flexible. For instance, to prove differentiability almost everywhere it is used that the measure is a (quotient of a) product of Lebesgue with a measure on the Cantor set and that the resulting measure is doubling. Of course, to find a measure zero universal differentiability set this should be possible for two measures which are mutually singular. For instance, for a three-part Cantor set one can aim to follow a similar construction assigning masses $a, b, c\in (0,1)$ to each component of the Cantor set where $a+b+c=1$. One would expect that Lipschitz functions are almost everywhere differentiable with respect to the (quotient of the) product of this measure with Lebesgue, and that if a similar construction is run with masses $a',b',c'$ then the resulting measures are singular if $a\neq a', b\neq b', c\neq c'$.

The organization of the paper is as follows. In Section \ref{prelim} we review relevant background, including the definition of Laakso space and the notion of derivatives in this context. In Section \ref{measures} we construct the measures and prove their main properties. In Section~\ref{RademacherProof} we show how the proof of Rademacher's theorem can be adapted from \cite{CPS22} and deduce the main results. We prove Theorem \ref{thm:PI} in Section \ref{sec:pi}.

\bigskip

\textbf{Acknowledgements:} The authors thank the anonymous reviewer for their careful reading of our manuscript and many insightful comments and suggestions.

A. Pinamonti is a member of {\em Gruppo Nazionale per l'Analisi Matematica, la Probabilit\`a e le loro Applicazioni} (GNAMPA), of the {\em Istituto Nazionale di Alta Matematica} (INdAM), and he was partially supported by MIUR-PRIN 2017 Project \emph{Gradient flows, Optimal Transport and Metric Measure Structures} and by the European Union under NextGenerationEU. PRIN 2022 Prot. n. 2022F4F2LH. 

G. Speight was partially supported by the National Science Foundation under Award No. 2348715 and partially supported by the Simons Foundation under Award No. 576219.

S. Eriksson-Bique was partially supported by the Finnish Research Council grant \#354241.

\section{Preliminaries}\label{prelim}
	The terminology and construction in this paper follow that of \cite{Laa00}. 
	Let $I=[0,1]$ and let $K\subset [0,1]$ the standard middle third Cantor set. Define $K_{0}:=(1/3)K$ and $K_{1}:=(1/3)K+(2/3)$ to be the left and right similar copies of $K$. We then define $K_{00}:=(1/3)K_0=(1/9)K$ and $K_{01}:=(1/3)K_1=(1/9)K+(2/9)$ to be the left and right similar copies of $K_{0}$. The set $K_{a}$ is defined similarly when $a$ is any finite string of $0$'s and $1$'s. We refer to such a string $a$ as a binary string.
	
	We define the height of a point $(x_{1}, x_{2})\in I\times K$ by $h(x_{1},x_{2}):=x_{1}$. If $n\in \bbN$ and $m_{i}\in \{0,1,2\}$ for $1\leq i\leq n$, we define $w(m_{1},\ldots, m_{n}):=\sum_{i=1}^{n} m_{i}/3^{i}$.
	A wormhole level of order $n$ is a set of the form
	\[\{w(m_{1},\ldots, m_{n})\} \times K \subset I\times K, \qquad m_{n}>0.\]
    We may also refer to the projection onto $I$, namely points $w(m_{1},\ldots, m_{n})\in I$ where $m_{n}>0$, as wormhole levels.
	
	\begin{definition}\label{def:wormhole}
		We define an equivalence relation $\sim$ on $I\times K$ as follows. For each $n\in \bbN$ and wormhole level $\{w(m_{1},\ldots, m_{n})\} \times K$ of order $n$, identify pairwise $\{w(m_{1},\ldots, m_{n})\} \times K_{a0}$ and $\{w(m_{1},\ldots, m_{n})\} \times K_{a1}$ for each binary string $a$ of length $n-1$. More precisely, a point $(x_1, x_2)\in \{w(m_{1},\ldots, m_{n})\} \times K_{a0}$ is identified with $(x_1, x_2+(2/3^n)) \in \{w(m_{1},\ldots, m_{n})\} \times K_{a1}$. Such an identified point is called a wormhole of order $n$.
	\end{definition}
	We denote the set of wormholes of order $n$ by $J_{n}:=\{w(m_{1},\ldots, m_{n}) : m_i \in \{0,1,2\}, m_n>0\} \times K$.
	Define $F:=(I\times K)/\sim$. Let $q\colon I\times K\to F$ be given by $q(x_1, x_2)=[x_1, x_2]$, where $[x_{1}, x_{2}]$ denotes the equivalence class in $F$ of $(x_{1},x_{2}) \in I\times K$. We define the height $h\colon F\to I$ by $h[x_{1},x_{2}]=x_{1}$. 
 We define a metric $d$ on $F$ by
	\[d(x,y)=\inf \{\mathcal{H}^{1}(p)\colon q(p) \mbox{ is a path joining }x \mbox{ and }y\},\]
	where $p\subset I\times K$. In \cite{Laa00} it is shown that any pair of points can be connected by a path and so the metric $d$ is well defined. The following proposition gives information about geodesics \cite[Proposition 1.1]{Laa00}.
	
	\begin{proposition}\label{prop_geodesic}
		Fix $x,y\in F$ with $h(x)\leq h(y)$. Let $[a,b]\subset I$ be an interval of minimum length that contains the heights of $x$ and $y$ and all the wormhole levels needed to connect those points with a path. Let $p$ be any path starting from $x$, going down to height $a$, then up to height $b$, then down to $y$. 
		
		Then $p$ is a geodesic connecting $x$ and $y$. All geodesics from $x$ to $y$ are of that form for some interval $[a',b']$ such that $b'-a'=b-a$.
	\end{proposition}

	Let $Q:=1+ (\log 2/\log 3)$. Note that $K$ is Ahlfors $(Q-1)$-regular. It is shown in \cite{Laa00} that $F$ is Ahflors $Q$-regular with respect to the metric $d$.
	
	\begin{definition}\label{def_laakso}
		The Laakso space is the set of equivalence classes $F:=(I\times K)/\sim$ equipped with the metric $d$ and Hausdorff dimension $\mathcal{H}^{Q}$.
	\end{definition} 
There are multiple ways that one can construct a Laakso space, and for simplicity we focus on the one particular construction that we gave, and call it the Laakso space. 

 Differentiability in $F$ is meant with respect to the Lipschitz chart $(F,h)$. This can be written explicitly as in the following definition.
	
	\begin{definition}\label{differentiability}
		Let $f\colon F\to \bbR$ and $x\in F$. We say that $f$ is differentiable at $x$ if there exists $Df(x)\in \bbR$ such that
		\[ \lim_{y\to x} \frac{f(y)-f(x)-Df(x)(h(y)-h(x))}{d(y,x)}=0.\]
	\end{definition}

The Laakso space is known to be a PI space \cite{Laa00}, hence admits a differentiable structure consisting of Lipschitz charts with respect to which Lipschitz functions are almost everywhere differentiable \cite{Che99}. However, these results do not give the charts explicitly. The following theorem was proved explicitly in \cite{CPS22} with the Lipschitz chart $(F,h)$, so the notion of differentiability is as in Definition \ref{differentiability}.
	
	\begin{theorem}\label{Rademacher}
		Every Lipschitz function $f\colon F\to \bbR$ is differentiable almost everywhere.
	\end{theorem}

 Note that it seems likely Theorem \ref{Rademacher} also follows from \cite[Chapter 9, Theorem 9.1]{CK15} on inverse limit spaces, once the Laakso space is recognized as such a space.

We will also make use of the directional derivatives defined below. As mentioned in \cite{CPS22}, this is a weaker requirement than being differentiable.

 \begin{definition}\label{def_vert_deriv}
		Let $f\colon F\to \bbR$ and $x=[x_{1},x_{2}]\in F$. 
		
		Suppose $x$ is not a wormhole. Whenever the limit exists, we define
		\begin{equation}\label{limit_vert_deriv}
			f_{I}(x):=\lim_{t\to 0} \frac{f[x_{1}+t,x_{2}]-f[x_{1},x_{2}]}{t}.
		\end{equation}
		The limit is one-sided if $x_{1}=0$ or $1$.
		
		Suppose $x$ is a wormhole of order $n$ and $(x_{1},x_{2})\in I\times K$ is the representative of $x$ with the smaller value of $x_{2}$. Whenever the limit exists, we define
		\begin{align*}
			f_{L}(x)&:=\lim_{t\to 0} \frac{f[x_{1}+t,x_{2}]-f[x_{1},x_{2}]}{t}\\
			f_{R}(x)&:=\lim_{t\to 0} \frac{f[x_{1}+t,x_2+(2/3^n)]-f[x_{1},x_2+(2/3^n)]}{t}.
		\end{align*}
		If $f_{L}(x)$ and $f_{R}(x)$ exist and are equal, we say that $f_{I}(x)$ exists and define it to be the common value. The limits are one-sided if $x_{1}=0,1$. 
	\end{definition}

\section{Singular Doubling Measures on $K$ and $F$}\label{measures}

\subsection{Measures on $K$}

Given $w\in (0,1)$ and a binary string $a$, let $p_{w}(K_{a}):=w^{s}(1-w)^{N-s}$ where $N$ is the length of the binary string $a$ and $s$ is the number of zeros in $a$. Intuitively, as the Cantor set is constructed by removing open middle thirds, $p_{w}$ assigns mass $w$ to the left similar copy and mass $1-w$ to the right similar copy.

\begin{proposition}\label{prop:measure}
For any $w\in (0,1)$, there is a unique Borel probability measure $\nu_{w}$ on $K$ so that $\nu_{w}(K_{a})=p_{w}(K_{a})$ for each binary string $a$. For any Borel set $E\subset K$, we have
\begin{equation}\label{nuformula} \nu_{w}(E)=\inf\left\{ \sum_{i=1}^{\infty}p_{w}(E_{i}): E\subset \bigcup_{i=1}^{\infty}E_{i}, \ E_{i} \mbox{ similar copies of }K\right\}.
\end{equation}
\end{proposition}

\begin{proof}
Let $\Sigma = \{0,1\}^\bbN$ be equipped with the product topology. Define a homeomorphism $\pi:\Sigma \to K$ by sending an infinite binary string $(a_i)_{i=1}^\infty$ to $\sum_{i=1}^\infty 2a_i 3^{-i} \in K$. Let $\nu_{w,\Sigma}=\prod_{i=1}^\infty \nu_0$ be a Bernoulli probability  measure on $\Sigma$, that is the infinite product measure of binary probability measures $\nu_0$ on $\{0,1\}$ where $\nu_0(\{0\})=w$ and $\nu_0(\{1\})=1-w$. Let $\nu_w =\pi_*(\nu_{w,\Sigma})$. Then, $\nu_{w}(K_{a})=p_{w}(K_{a})$ follows directly from the definition of $\nu_{w}$. Equation \eqref{nuformula}follows since the similar copies of $K$ generate the Borel $\sigma$-algebra of $K$, and uniqueness of $\nu_{w}$ follows for the same reason.

\end{proof}

We equip $K$ with the induced Euclidean metric. Recall that a measure $m$ on a metric space $(X,d)$ is doubling if there is a constant $D>1$ so that for all balls $B(x,r)$ in $X$ it holds that $m(B(x,2r))\leq D \mu(B(x,r))$.

\begin{proposition}\label{nudoubling}
The probability measure $\nu_{w}$ is doubling on $K$ for any $w\in (0,1)$.
\end{proposition}

\begin{proof}
Let $x\in K$ and $0<r\leq 1/9$. Choose integer $N\geq 3$ with $\frac{1}{3^{N}}<r\leq \frac{1}{3^{N-1}}$. Notice $B(x,r)$ must contain a similar copy of the Cantor set $K_{a}$ which contains $x$ and where the binary string $a$ is of length $N$. If $s$ is the number of zeros in $a$, then we have
\[\nu(B(x,r))\geq \nu(K_{a})=w^{s}(1-w)^{N-s}.\]
On the other hand, $2r\leq \frac{2}{3^{N-1}}<\frac{1}{3^{N-2}}$. Since similar copies of level $N-2$ are separated by a distance $1/3^{N-2}$, this implies $B(x,2r)$ is contained inside $K_{b}$ where $b$ is the binary string equal to $a$ except with the last two entries deleted. Hence
\[\nu(B(x,2r))\leq \nu(K_{b})\leq w^{s-2}(1-w)^{N-s-2}.\]
This implies $\nu(B(x,2r))\leq w^{-2}(1-w)^{-2}\nu(B(x,r))$. Hence $\nu$ is doubling.
\end{proof}

We next recall the strong law of large numbers from probability theory \cite{Bil76}. Recall that in the context of a probability space $(\Omega,\Sigma,P)$, where $\Sigma$ is a $\sigma$-algebra and $P$ is a probability measure, a random variable $X:\Omega\to [-\infty,\infty]$ is an extended real-valued measurable function on $X$. The mean or expectation $E(X)$ is simply the integral of $X$ with respect to $P$ if it exists. An event holds with probability $1$ if it holds almost surely. 

The distribution function of a random variable $X$ is $F\colon \mathbb{R}\to [0,1]$ defined by $F(x)=P(X\leq x):=P(\{w:X(w)\leq x\})$. A collection of random variables is identically distributed if they have the same distribution function. 

A finite collection of random variables $X_{1},\ldots, X_{k}$ is independent if
\[P(X_{1}\leq x_{1},\cdots, X_{k}\leq x_{k})=P(X_{1}\leq x_{1})\cdots P(X_{k}\leq x_{k}).\]
An infinite collection of random variables is independent if each finite subcollection is independent. 

\begin{lemma}[Strong Law of Large Numbers]\label{stronglaw}
Let $X_{1}, X_{2}, \ldots$ be random variables on a probability space. Assume they are independent and identically distributed and have finite mean. Then $\frac{1}{n}\sum_{k=1}^{n}X_{k}\to E(X_{1})$ with probability $1$.
\end{lemma}

For each integer $n\geq 1$, define $X_{n}\colon K\to \mathbb{R}$ as follows. If $n=1$, then $X_{1}$ is identically $1$ on $K_{0}$ and identically $0$ on $K_{1}$. If $n>1$, then $X_{n}$ is identically $1$ on $K_{a0}$ and identically $0$ on $K_{a1}$ for any binary string $a$ of length $n-1$. Clearly $X_{n}$ is Borel measurable for each $n\geq 1$. Denote $S_{n}=\sum_{k=1}^{n}X_{k}$. For each $w\in (0,1)$, define the Borel set
\begin{equation}\label{Ew}
E_{w}=\{x\in K: S_{n}(x)/n \to w\}.
\end{equation}

\begin{proposition}\label{nusingular}
For any $w\in (0,1)$, we have $\nu_{w}(K\setminus E_{w})=0$ and $\nu_{w'}(E_{w})=0$ for all $w'\in (0,1)\setminus \{w\}$.

In particular, for all $w,w'\in (0,1)$ with $w\neq w'$, the probability measures $\nu_{w}$ and $\nu_{w'}$ are mutually singular.
\end{proposition}

\begin{proof}

Fix $w\in (0,1)$. Recall the construction of $\nu_w$ from the proof of Proposition~\ref{prop:measure}. The measure $\nu_{w,\Sigma}$ is a Bernoulli probability measure and the random variables $Y_i:=X_i \circ \pi$ are of the form $1-Z_{i}$, where $Z_{i}$ are the independent and identically distributed projections onto the $i$'th component of $\Sigma$. In particular, $Y_{i}$ are independent and identically distributed. Hence, in $\Sigma$, $\frac{1}{n}\sum_{i=1}^n Y_i\to w$ almost surely with respect to $\nu_{w,\Sigma}$ by Lemma \ref{stronglaw}. Since $\pi$ is a measure preserving bijection, also $\frac{1}{n}\sum_{i=1}^n X_i(x)\to w$ for $\nu_{w}$-almost every $x\in K$. Thus, $\nu_w(E_w)=1$. Further, by the same argument, $\nu_{w'}(E_w)=0$ for any $w'\in (0,1)\setminus\{w\}$. Thus, $\nu_{w}$ and $\nu_{w'}$ are pairwise singular.

\end{proof}

\subsection{Measures on $F$}

Recall that $\mathcal{H}^{1}$ denotes the Hausdorff measure on $I$ with respect to the Euclidean distance. For any $w\in (0,1)$, we define
\begin{equation}\label{defmuw}
\mu_{w}=q_{\ast}(\mathcal{H}^{1}\times \nu_{w})
\end{equation}
and
\begin{equation}\label{Nw}
N_{w}=q(I\times E_{w}).
\end{equation}

 To prove that $\mu_{w}$ is doubling, the following simple lemma will be useful. 
	
	\begin{lemma}\label{atmostone}
		Suppose $x,y\in F$ with $d(x,y)<1$. Let $N\geq 1$ be the unique integer satisfying $1/3^{N}\leq d(x,y)< 1/3^{N-1}$. Then any geodesic joining $x$ to $y$ can pass through at most one wormhole of level less than or equal to $N-1$.

  In particular, suppose $x\in F$ and $0<r<1$. Let $N\geq 1$ be the unique integer satisfying $1/3^{N}\leq r< 1/3^{N-1}$. Then at vertical distance at most $r$ above and below $x$, one can find at most one wormhole with a level less than or equal to $N-2$.
	\end{lemma}

 \begin{proof}
The first part of the Lemma was proved in \cite{CPS22}. To prove the second part notice that $(h(x)-r,h(x)+r)$ has length $2r$ and
\[2r<2/3^{N-1}<1/3^{N-2}.\]
Since wormholes of level less than or equal to $N-2$ are spaced apart by a distance $1/3^{N-2}$, the conclusion follows.
 \end{proof}

\begin{proposition}
For every $w\in (0,1)$, $\mu_{w}$ is a doubling measure with respect to the  metric $d$ on $F$.
\end{proposition}

\begin{proof}
We denote $\nu=\nu_{w}$ and $\mu=\mu_{w}$ for convenience. That $\mu$ is Borel follows from continuity of $q$. Fix $x=[x_{1},x_{2}]\in F$ and $0<r<1/3$. Fix an integer $N$ such that $1/3^{N}\leq r<1/3^{N-1}$.
		
		We  estimate $\mu(B(x,r))$ from below. Without loss of generality assume $x_{1}<2/3$, since otherwise one can apply a similar argument with upwards and downwards reversed. For each $M\geq 1$, wormholes of level $M$ are spaced apart by a distance at most $2/3^{M}$. If $M\geq N+2$ then $r/2\geq 2/3^{M}$. Hence, starting at $x$, one can reach by a curve of length at most $r$ any point $y=[y_{1},y_{2}]$ satisfying both:
		\begin{itemize}
			\item $x_{1}\leq y_{1}\leq x_{1}+r/2$, and
			\item $y_{2}$ is reached from $x_{2}$ by wormholes of level $M\geq N+2$.
		\end{itemize}
		Note that if $x$ is a wormhole level then either representative of $x_{2}$ may be used here. This shows that $q^{-1}(B(x,r))$ contains a set of the form $[x_{1},x_{1}+r/2]\times K_{N+1}$,  where $K_{N+1}\subset K$ is a similar copy of $K$ obtained after splitting $N+1$ times which contains $x_{2}$. Hence
		\begin{align*}
			\mu(B(x,r))&=(\mathcal{H}^{1}\times \nu)(q^{-1}(B(x,r))\\
			&\geq (r/2)\nu(K_{N+1}).
		\end{align*}
In particular, balls have strictly positive measure.

    	We next estimate $\mu(B(x,2r))$ from above. Notice $2r<1/3^{N-2}$. By Lemma~\ref{atmostone}, at vertical distance at most $2r$ above and below $x$, one can find at most one wormhole with a level less than or equal to $N-3$. Hence $q^{-1}(B(x,2r))$ is contained in a set of the form
		\[\Big( [x_{1}-2r,x_{1}+2r]\times K_{N-2}^{1}\Big) \cup \Big([x_{1}-2r,x_{1}+2r]\times K_{N-2}^{2}\Big),\]
		where $K_{N-2}^{1}, K_{N-2}^{2}$ are similar copies of $K$ obtained after splitting $N-2$ times. Note that one of the two similar copies (temporarily denoted $K_{a}$ for some binary string $a$) contains $x_{2}$, while the other is obtained by switching one of the entries of $K_{a}$ at a coordinate less than or equal to $N-3$. This leads to the estimate
		\begin{align*}
			\mu(B(x,2r))&=(\mathcal{H}^{1}\times \nu)(q^{-1}(B(x,2r)))\\
			&\leq 4r(\nu(K_{N-2}^{1})+\nu(K_{N-2}^{2})).
		\end{align*}

  Combining the two estimates yields
  \[\frac{\mu(B(x,2r))}{\mu(B(x,r))}\leq 8 \frac{(\nu(K_{N-2}^{1})+\nu(K_{N-2}^{2}))}{\nu(K_{N+1})}.\]
  The result follows because $\nu(K_{N-2}^{1})/\nu(K_{N+1})$ and $\nu(K_{N-2}^{2})/\nu(K_{N+1})$ are both bounded above by $\max(w^{-4}, (1-w)^{-4})$.  
\end{proof}

\begin{remark}
If $\nu$ is an arbitrary doubling measure on $K$, it does not necessarily follow that $\mu=q_{\ast}(\mathcal{H}^{1}\times \nu)$ is a doubling measure on $F$. For instance, given $0<\lambda<1$ and $0<\widehat{\lambda}<1$, define a measure $\nu$ on $K$ as follows. First assign measure $1/2$ to both the left and right similar copies $K_{0}$ and $K_{1}$ of $K$. Then, at any given stage, if $K_{a}$ is a similar copy for which $a$ starts with $0$ we assign a proportion $\lambda$ of the measure to $K_{a0}$ and $1-\lambda$ to $K_{a1}$, while if $K_{a}$ is a similar copy for which $a$ starts with $1$ we assign a proportion $\widehat{\lambda}$ of the measure to $K_{a0}$ and $1-\widehat{\lambda}$ to $K_{a1}$. Then, using a similar argument to that of Proposition \ref{nudoubling}, it is not difficult to see $\nu$ is a doubling measure on $K$. We claim $\mu:=q_{\ast}(\mathcal{H}^{1}\times \nu)$ is not doubling on $F$. To see this we consider for any $m\geq 1$ the open balls 
\[B_{m}=B\left(\left[\frac{1}{3}+\frac{1}{3^m},0\right],\frac{1}{3^{m}}\right), \qquad 2B_{m}=B\left(\left[\frac{1}{3}+\frac{1}{3^m},0\right],\frac{2}{3^{m}}\right).\]
Then $p^{-1}(B_{m})$ is contained in a set of the form $(1/3,1/3+2/3^{m})\times K_{b}$ where $|b|=m-1$ and $b$ begins with a $0$. Hence
\[\mu(B_{m})\leq \frac{2}{3^{m}}\frac{1}{2}\lambda^{m-2}=\frac{\lambda^{m-2}}{3^{m}}.\]
On the other hand, $p^{-1}(2B_{m})$ contains a set $(1/3-1/3^{m},1/3+1/3^{m})\times (K_{b}\cup K_{b'})$ where $b'$ agrees with $b$ except the first entry is $1$ rather than $0$. Hence
\[\mu(2B_{m})\geq \frac{2}{3^{m}}\left(\frac{1}{2}\lambda^{m-2}+\frac{1}{2}\widehat{\lambda}^{m-2}\right)=\frac{1}{3^{m}}\left(\lambda^{m-2}+\widehat{\lambda}^{m-2}\right).\]
Hence
\[ \frac{\mu(2B_{m})}{\mu(B_{m})}\geq 1+\left( \frac{\widehat{\lambda}}{\lambda}  \right)^{m-2}.\]
If $\widehat{\lambda}>\lambda$, letting $m\to \infty$ shows $\mu$ is not doubling. A similar argument applies if instead $\widehat{\lambda}<\lambda$, changing the center of the balls to the point $[\frac{1}{3}+\frac{1}{3^m},1]$.
\end{remark}

\begin{proposition}\label{singular}
For any $w\in (0,1)$, we have $\mu_{w}(F\setminus N_{w})=0$ and $\mu_{w'}(N_{w})=0$ for all $w'\in (0,1)\setminus \{w\}$.

In particular, for all $w,w'\in (0,1)$ with $w\neq w'$, the probability measures $\mu_{w}$ and $\mu_{w'}$ are mutually singular.
\end{proposition}

\begin{proof}
Fix any $w,w'\in (0,1)$ with $w\neq w'$. By Proposition \ref{nusingular}, we know that $\nu_{w}(K\setminus E_{w})=\nu_{w'}(E_{w})=0$. Recall $N_{w}=q(I\times E_{w})\subset F$. Note that the symmetric difference of $q^{-1}(q(I\times E_{w}))$ and $I\times E_{w}$ is contained in a set of the form $C\times K$ where $C$ is countable, hence has measure zero with respect to $\mathcal{H}^{1}\times \nu_{w}$. Hence
\begin{align*}
\mu_{w}(F\setminus N_{w})&=(\mathcal{H}^{1}\times \nu_{w})((I\times K)\setminus q^{-1}(q(I\times E_{w})))\\
&= (\mathcal{H}^{1}\times \nu_{w})((I\times K)\setminus (I\times E_{w}))\\
&= (\mathcal{H}^{1}\times \nu_{w})(I\times (K\setminus E_{w}))\\
&=\nu_{w}(K\setminus E_{w})\\
&=0.
\end{align*}
On the other hand, we have
\begin{align*}
\mu_{w'}(N_{w})&= (\mathcal{H}^{1}\times \nu_{w'})(q^{-1}(q(I\times E_{w})))\\
&= (\mathcal{H}^{1}\times \nu_{w'})(I\times E_{w})\\
&= \nu_{w'}(E_{w})\\
&=0.
\end{align*}
This proves the first part of the proposition. The second is then an immediate consequence.
This concludes the proof.
\end{proof}

We now briefly describe how $\mu_{1/2}$ is related to $\mathcal{H}^{Q}$ which is the measure normally used on $F$. Note, since $K$ is $(Q-1)$-Ahlfors regular, we have $0<\mathcal{H}^{Q-1}(K)<\infty$.

\begin{proposition}\label{consistent}
For every Borel set $E\subset K$ we have \[\nu_{1/2}(E)=\frac{\mathcal{H}^{Q-1}(E)}{\mathcal{H}^{Q-1}(K)}.\]
Hence for every Borel set $E\subset F$ we have
\[\mu_{1/2}(E)=\frac{q_{\ast}(\mathcal{H}^{1}\times \mathcal{H}^{Q-1})(E)}{\mathcal{H}^{Q-1}(K)}.\]
Consequently, there exists a constant $C\geq 1$ such that for every Borel set $E\subset F$
\[C^{-1}\mathcal{H}^{Q}(E)\leq \mu_{1/2}(E)\leq C\mathcal{H}^{Q}(E),\]
where $\mathcal{H}^{Q}$ denotes the Hausdorff measure of dimension $Q$ on $F$ with metric $d$.
\end{proposition}

\begin{proof}
For any Borel set $E\subset K$, denote $\widetilde{\nu}(E)=\frac{\mathcal{H}^{Q-1}(E)}{\mathcal{H}^{Q-1}(K)}$. Then $\widetilde{\nu}$ is a probability measure on Borel subsets of $K$. It also holds that $\widetilde{\nu}(K_{a})=(1/2)^{N}=p_{1/2}(K_{a})$ for any binary string $a$ of length $N$. Since $\nu_{1/2}$ was a unique extension of $p_{1/2}$, it follows that $\widetilde{\nu}(E)=\nu_{1/2}(E)$ for any Borel set $E\subset K$. Hence the first part of the proposition follows.

The second part follows by definition of product measure and the definition of $\mu_{w}$ in \eqref{defmuw}. The third part follows by combining the second part with the fact that $q_{\ast}(\mathcal{H}^{1}\times \mathcal{H}^{Q-1})$ is bounded within constant multiples of $\mathcal{H}^{Q}$ since both are Ahlfors $Q$-regular, as explained in \cite{CPS22}. 
\end{proof}

\section{Rademacher's Theorem for a Singular Measure}\label{RademacherProof}

In this section we prove Theorem \ref{rademachermuw}. Before doing so, we describe how it can be combined with the results of the previous section to prove our main result Theorem~\ref{thmuds}.

   \begin{proof}[Proof of Theorem \ref{thmuds} from Theorem \ref{rademachermuw}]
Fix any $w\neq 1/2$ and consider the Borel set $N_{w}$. By Proposition \ref{singular}, $\mu_{w}(N_{w})>0$. Hence, by Theorem \ref{rademachermuw}, each Lipschitz function $f\colon F\to \bbR$ is differentiable at some point of $N_{w}$ (with the point possibly depending on $f$). Since $w\neq 1/2$, applying Proposition \ref{singular} implies $\mu_{1/2}(N_{w})=0$. Hence, by Proposition \ref{consistent}, $\mathcal{H}^{Q}(N_{w})=0$.
  \end{proof}
	
We use the rest of this section to prove Theorem \ref{rademachermuw}. We divide the proof into several steps, following \cite{CPS22} with adjustments to account for the fact $\mu_{w}$ is doubling instead of $Q$-Ahlfors regular. For the remainder of this section fix $w\in (0,1)$ and denote $\nu=\nu_{w}$, $\mu=\mu_{w}$.
 
	\subsection{Measure Theoretic Preliminaries}

The following lemma follows by Tonelli's theorem. The proof is the same as in \cite{CPS22}, up to replacing $\mathcal{H}^{Q}$ with $\mu=q_{\ast}(\mathcal{H}^{1}\times \nu)$.
 
	\begin{lemma}\label{weakFubini}
		Suppose $A\subset F$ is Borel with respect to the metric $d$ and 
		\[\mathcal{L}^{1}\{t\in I : [t,z]\in A\}=0 \mbox{ for every }z\in K.\]
		Then $\mu(A)=0$.
	\end{lemma}

 The next lemma is as in \cite{CPS22}, except the measure $\mathcal{H}^{Q}$ is replaced by $\mu$. The proof is the same, since we may apply Lemma \ref{weakFubini} with the measure $\mu$ instead of $\mathcal{H}^{Q}$.
	
	\begin{lemma}\label{firstmeas}
		The following statements hold for every Lipschitz map $f\colon F\to \mathbb{R}$.
		\begin{enumerate}
			\item For every $z\in K$, the set
			\[D_{z}:=\{t\in I: \mbox{the directional derivative }f_{I}[t,z] \mbox{ exists}\}\]
			is Borel with respect to the Euclidean metric on $I$ and has full $\mathcal{L}^{1}$ measure.
			\item For every $z\in K$, the map from $D_{z}$ to $\bbR$ defined by $t\mapsto f_{I}[t,z]$ is Borel measurable with respect to the Euclidean metric on $I$.
			\item The set
			\[D:=\{x\in F: \mbox{the directional derivative }f_{I}(x) \mbox{ exists}\}\]
			is Borel measurable with respect to $d$ on $F$ and has full $\mu$ measure.
			\item The map $f_{I}\colon D\to \bbR$ defined by $x\mapsto f_{I}(x)$ is Borel measurable.
		\end{enumerate}	
	\end{lemma}

	Recall that if $(X,m)$ is a doubling metric measure space than the Lebesgue density theorem holds and Borel functions are approximately continuous almost everywhere. I.e. if $g\colon X\to \bbR$ is Borel, then for almost every $x\in X$
	\[ \lim_{r\to 0} \frac{m\{y\in B(x,r): |g(y)-g(x)|>\varepsilon\}}{m(B(x,r))}=0 \qquad \mbox{for every }\varepsilon>0. \]
 	We will use these facts in $\bbR$ equipped with Euclidean distance and Lebesgue measure and in $F$ equipped with the metric $d$ and measure $\mu$.

	\subsection{Auxilliary Sets}

	Let $f\colon F\to \mathbb{R}$ be a Lipschitz function and let $D\subset F$ denote the set of points where the directional derivative of $f$ exists. Denote $L=\mathrm{Lip}(f)$. Let $C_\mu\geq 1$ be the doubling constant of $\mu$. By iterating the doubling condition, there exists $Q>0$ and $C_{Q}\geq 1$ both depending only on $C_{\mu}$ such that for all $x\in F$, $0<r\leq R$, and $y\in B(x,R)$, we have
 \[ \frac{\mu(B(y,r))}{\mu(B(x,R))}\geq C_{Q}^{-1}\left( \frac{r}{R} \right)^{Q}.    \]
Note that wormholes have measure zero with respect to $\mu$. To see this notice that if $W\subset F$ is the set of wormholes, then $q^{-1}(W)$ is contained in a set of the form $C\times K$ where $C$ is countable. Hence
\begin{align*}
\mu(W)&=(\mathcal{H}^{1}\times \nu)(q^{-1}(W))\\
&\leq (\mathcal{H}^{1}\times \nu)(C\times K)\\
&= \mathcal{H}^{1}(C)\\
&=0.
\end{align*}
	\begin{definition}
		Let $D'$ be the set of all points $x\in D$ which are not a wormhole. We define several sets as follows.
		\begin{enumerate}
			\item For each $\varepsilon>0$ and $x\in D$,
			\[D_{\varepsilon}(x):=\{ y\in D:|f_{I}(y)-f_{I}(x)|\leq \varepsilon\}.\]
			\item For each $\varepsilon>0$ and integer $k\geq 1$, let $E_{k}^{1}(\varepsilon)$ be the collection of all points $x=[x_{1},x_{2}]\in D'$ such that
			\[\mathcal{L}^{1}\{t\in (x_{1}-r,x_{1}+r)\cap I:[t,x_{2}]\notin D_{\varepsilon}(x)\}\leq \varepsilon r\]
			for every $0<r<1/k$.
			\item For each $\varepsilon>0$ and integer $k\geq 1$, let $E_{k}^{2}(\varepsilon)$ be the collection of all points $x\in D'$ for which
			\begin{equation}\label{E2ineq}
			    \mu\Big( B(x,r)\setminus ( D_{\varepsilon}(x) \cap E_{k}^{1}(\varepsilon) ) \Big) \leq \frac{C_{Q}^{-1}\varepsilon^{Q}\min(w,1-w)}{2^{Q+2}} \mu(B(x,r))
       	\end{equation}
			for every $0<r<1/k$.
		\end{enumerate}

	\end{definition}

The proof of the next lemma is the same as in \cite{CPS22} with $\mathcal{H}^{Q}$ replaced by $\mu$, which is possible by applying Lemma \ref{weakFubini}.
	
	\begin{lemma}\label{E1}
		For all $\varepsilon>0$ and integer $k\geq 1$, $E_{k}^{1}(\varepsilon)$ is Borel with respect to $d$ and
		\[\mu\left( F\setminus \bigcup_{k=1}^{\infty} E_{k}^{1}(\varepsilon) \right) =0.\]	
	\end{lemma}

 The proof of the next lemma requires minor adaptations from \cite{CPS22} to account for the change in measure. We first make two remarks.
 
 \begin{remark}\label{measurecontinuity}
 First we claim that $\lim_{s\to s_{0}}\mu(B(x,s))=\mu(B(x,s_{0}))$ for any $s_{0}>0$. To see this, note $\mu$ is a doubling measure on a length space so satisfies a $\delta$-annular decay property for some $0<\delta\leq 1$ depending only on the doubling constant of $\mu$ \cite{Buc99}. To be more specific, there is $K\geq 1$ such that for all $x\in F$, $r>0$, $0<\varepsilon<1$,
 \[ \mu(B(x,r)\setminus B(x,r(1-\varepsilon))) \leq K\varepsilon^{\delta}\mu(B(x,r)). \]
 From this, the claim clearly follows.

 Second we claim that $\lim_{x\to x_{0}}\mu(B(x,s))=\mu(B(x_{0},s))$ for any $x_{0}\in F$. To see this notice that $|\mu(B(x,s))-\mu(B(x_{0},s))|$ is bounded by the maximum of
\[\mu(B(x_{0},s+d(x,x_{0})))-\mu(B(x_{0},s))\]
and
\[\mu(B(x,s+d(x,x_{0})))-\mu(B(x,s)).\]
Both of these converge to zero as $x\mapsto x_{0}$, so the claim follows.
 \end{remark}
	
	\begin{lemma}\label{E2}
		For $\varepsilon>0$ and integer $k\geq 1$, $E_{k}^{2}(\varepsilon)$ is Borel with respect to $d$ and
		\[\mu\left( F\setminus \bigcup_{k=1}^{\infty} E_{k}^{2}(\varepsilon) \right) =0.\]	
	\end{lemma}
	
	\begin{proof} 
		Fix $\varepsilon>0$. We first show that $E_{k}^{2}(\varepsilon)$ is Borel with respect to $d$. Note that \eqref{E2ineq} holds for all $0<r<1/k$ if and only if it holds for all rational $0<r<1/k$. This follows by choosing a rational sequence $0<r_{n}<1/k$ with $r_{n}\downarrow r$, applying \eqref{E2ineq} for each $n$, and using Remark \ref{measurecontinuity}. Hence it suffices to show that both sides of the estimate defining $E_{k}^{2}(\varepsilon)$ are Borel measurable functions of $x$. For the right side we simply note that $x\mapsto \mu(B(x,r))$ is continuous for each $r$ by Remark \ref{measurecontinuity}. For the left side, consider $D' \to \bbR$ given by $x\mapsto \mu(B(x,r)\setminus (D_{\varepsilon}(x)\cap E_{k}^{1}(\varepsilon)))$. Notice that for every $\alpha>0$, the set
		\begin{align*}
			\left\{x\in D': \mu\{y\in B(x,r): y\notin E_{k}^{1}(\varepsilon) \mbox{ or }|f_{I}(y)-f_{I}(x)|>\varepsilon\}>\alpha\right\}
		\end{align*}
		can be written as 
		\begin{align*}
			& \bigcup_{\substack{\eta>\varepsilon\\ \eta \in \mathbb{Q}}} \bigcap_{n=1}^{\infty}\bigcup_{q\in \mathbb{Q}} \Big(\{x\in D': |f_{I}(x)-q|<1/n\} \\
			& \qquad \qquad \qquad \cap \{x\in D': \mu\{y\in B(x,r): |f_{I}(y)-q|>\eta \mbox{ or }y\notin E_{k}^{1}(\varepsilon)\}>\alpha\} \Big).
		\end{align*}
		 The first set inside the decomposition above is Borel by Lemma \ref{firstmeas}. The second is an open subset of $D'$, using Remark \ref{measurecontinuity}, hence Borel. Hence $E_{k}^{2}(\varepsilon)$ is Borel.
		
		Using Lemma \ref{E1}, almost every point of $F$ is a density point of $E_{k}^{1}(\varepsilon)$ with respect to $\mu$ for some $k\geq 1$. Also, since $f_{I}$ is Borel, almost every point of $F$ is a point of approximate continuity of $f_{I}$ with respect to $\mu$. Hence for almost every $x$ there exists $k\in \bbN$ and $R>0$ such that
		\[\mu\Big( B(x,r)\setminus ( D_{\varepsilon}(x) \cap E_{k}^{1}(\varepsilon) ) \Big) < \frac{C_{Q}^{-1}\varepsilon^{Q}\min(w,1-w)}{2^{Q+2}} \mu(B(x,r))\]
		for every $0<r<R$. Choose $K\in \bbN$ such that $K\geq k$ and $1/K<R$. Then using the fact $E_{k}^{1}(\varepsilon)\subset E_{K}^{1}(\varepsilon)$ it follows
		\[\mu\Big( B(x,r)\setminus ( D_{\varepsilon}(x) \cap E_{K}^{1}(\varepsilon) ) \Big) <\frac{C_{Q}^{-1}\varepsilon^{Q}\min(w,1-w)}{2^{Q+2}} \mu(B(x,r))\]
		for every $0<r<1/K$. Hence $x\in E_{K}^{2}(\varepsilon)$. This shows $\mu(F\setminus \bigcup_{k=1}^{\infty} E_{k}^{2}(\varepsilon)) =0$.
	\end{proof}
	
	\subsection{Choice of Suitable Line Segments}
	
	Fix $0<\varepsilon<1$. Let $x\in \bigcup_{k=1}^{\infty} E_{k}^{2}(\varepsilon)$ and fix $K\geq 1$ such that $x\in E_{K}^{2}(\varepsilon)$. Let $y\in F$ with $d(y,x)<1/(2K)$. Let $N\geq 1$ be the unique integer such that $1/3^{N}\leq d(x,y)< 1/3^{N-1}$.
	
	Assume that infinitely many wormhole levels are needed to connect $x$ to $y$ by a geodesic. It will be clear how the following argument can be simplified if only finitely many wormhole levels or even no wormhole levels are needed. Denote $T=d(x,y)$ and choose $\gamma\colon [0,T]\to F$ such that
	\begin{itemize}
		\item $\gamma$ is a geodesic from $x$ to $y$ with $\gamma(0)=x$ and $\gamma(T)=y$.
		\item $\gamma$ is a concatenation of countably many lines in the $I$ direction parameterized at unit speed.
	\end{itemize}
	By Lemma \ref{atmostone}, any geodesic joining $x$ to $y$ must pass through at most one wormhole of level less than or equal to $N-1$. We enumerate the wormhole levels needed to connect $x$ to $y$ by a strictly increasing sequence $N_{i}$ for integer $i\geq 0$, where possibly $N_{0}\leq N-1$, but necessarily $N_{i}\geq N$ for $i\geq 1$. Since $N_{1}\geq N$ and $N_{i}$ are strictly increasing, it follows that $N_{i}\geq N+i-1$ for $i\geq 1$.
	
	For each $i\geq 0$, let $\lambda_{i}$ be the point in the interval $[0,T]$ where $\gamma$ jumps using the wormhole of order $N_{i}$. Geodesics in $F$ can be chosen so that they change their direction (up or down) in the $I$ component at most twice (Proposition \ref{prop_geodesic}). Hence, during any subinterval of $[0,T]$ of length $t$, the geodesic spends at least a time $t/3$ following the same direction (either up or down but not changing between them) in the $I$ component. Since in any direction wormhole levels of order $N_{i}$ are spaced apart by at most a distance $2/3^{N_{i}}$, we can additionally choose $\gamma$ so that it satisfies:
	\begin{itemize}
		\item $\lambda_{0}\leq d(x,y)$, and
		\item $\lambda_{i}\leq 2/3^{N_{i}-1}$ for $i\geq 1$.
	\end{itemize}
	Using $N_{i}\geq N+i-1$ for $i\geq 1$ and the definition of $N$, we estimate as follows:
	\begin{align*}\sum_{i=0}^{\infty} \lambda_{i} &\leq d(x,y)+\sum_{i=1}^{\infty} \frac{2}{3^{N_{i}-1}}\\
		&\leq d(x,y)+\sum_{i=1}^{\infty}\frac{2}{3^{N+i-2}}\\
		&=d(x,y)+\frac{1}{3^{N-2}}\\
		&\leq 10d(x,y).	
	\end{align*}
	Let $(\mu_{i})_{i=0}^{\infty}$ be a strictly decreasing rearrangement of $\{\lambda_{i}:i\geq 0\}\cup \{T\}$. Thus $\mu_{0}=T$, $\mu_{i}\to 0$ as $i\to \infty$, $\gamma|_{[\mu_{i+1},\mu_{i}]}$ is a line segment for each $i\geq 0$, and
	\begin{equation}\label{summu}
		\sum_{i=0}^{\infty} \mu_{i}=\sum_{i=0}^{\infty} \lambda_{i} +T \leq 11d(x,y).
	\end{equation}
	Denote $p_{i}=\gamma(\mu_{i})$ for $i\geq 0$. Notice $p_{0}=y$ and $p_{i}\to x$ as $i\to \infty$. It follows that 
	\begin{equation}\label{telescope}
		f(y)-f(x)=\sum_{i=0}^{\infty} (f(p_{i})-f(p_{i+1})).
	\end{equation}
	Since $\gamma|_{[\mu_{i+1},\mu_{i}]}$ is a line segment in the $I$ direction, it follows $p_{i}$ is reached from $p_{i+1}$ by travelling a displacement $h(p_{i})-h(p_{i+1})$ in the $I$ direction.
	
	\subsection{Estimate Along Line Segments}
	
	Our aim is to show that $f(p_{i})-f(p_{i+1})$ is well approximated by $f_{I}(x)(h(p_{i}) - h(p_{i+1}))$ for every $i\geq 0$. Fix $i\geq 0$ until otherwise stated.
	
	\begin{lemma}\label{qs}
		There exist points $q_{i}, q_{i+1}\in F$ with the following properties:
		\begin{enumerate}
			\item $d(q_{i+1},p_{i+1})\leq \varepsilon \mu_{i+1}$,
			\item $d(q_{i},p_{i})\leq 6\varepsilon \mu_{i+1}$,
			\item $q_{i+1}\in E_{K}^{1}(\varepsilon)\cap D_{\varepsilon}(x)$,
			\item $q_{i}$ is reached from $q_{i+1}$ by travelling a vertical displacement $h(p_{i})-h(p_{i+1})$ in the $I$ direction.
		\end{enumerate}
	\end{lemma}
	
	\begin{proof}
		Using $0<\varepsilon<1$ and $\mu_{i+1}\leq T<1/2K$ gives $\mu_{i+1}+\varepsilon\mu_{i+1}<1/K$. Hence the fact that $x\in E_{K}^{2}(\varepsilon)$ gives,
		\[\mu \Big( B(x,\mu_{i+1}+\varepsilon\mu_{i+1})\setminus ( D_{\varepsilon}(x) \cap E_{K}^{1}(\varepsilon)) \Big) <\frac{C_{Q}^{-1}\varepsilon^{Q}\min(w,1-w)}{2^{Q+2}}\mu(B(x,\mu_{i+1}+\varepsilon\mu_{i+1})).\]
		Since $\gamma$ is a geodesic, $d(x,p_{i+1})=d(\gamma(0),\gamma(\mu_{i+1}))=\mu_{i+1}$. Hence
		\[B(p_{i+1},\varepsilon \mu_{i+1}) \subset B(x,\mu_{i+1}+\varepsilon\mu_{i+1}).\]
  It follows that
		\[\mu \Big( B(p_{i+1},\varepsilon \mu_{i+1}) \setminus ( D_{\varepsilon}(x) \cap E_{K}^{1}(\varepsilon)) \Big) <\frac{C_{Q}^{-1}\varepsilon^{Q}\min(w,1-w)}{2^{Q+2}}\mu(B(x,\mu_{i+1}+\varepsilon\mu_{i+1})).\]
  However, using the doubling property and $0<\varepsilon<1$,
  \[\frac{ \mu(B(p_{i+1},\varepsilon \mu_{i+1})) }{\mu(B(x,\mu_{i+1}+\varepsilon\mu_{i+1}))}\geq C_{Q}^{-1}\left( \frac{\varepsilon}{1+\varepsilon} \right)^{Q}> C_{Q}^{-1}\frac{\varepsilon^{Q}}{2^{Q}}.\]
Combining the previous two steps gives 
\begin{equation}\label{J5a}
 \mu \Big( B(p_{i+1},\varepsilon \mu_{i+1}) \setminus ( D_{\varepsilon}(x) \cap E_{K}^{1}(\varepsilon)) \Big) < \frac{\min(w,1-w)}{4}\mu(B(p_{i+1},\varepsilon \mu_{i+1})).
 \end{equation}

Fix an integer $B\geq 1$ such that $1/3^B\leq \varepsilon \mu_{i+1} < 1/3^{B-1}$. This implies that within a vertical distance $\varepsilon \mu_{i+1}$ of $p_{i+1}$, there is at most two wormhole levels of order less than or equal to $B-1$. Let $S$ be the set of points $z\in B(p_{i+1},\varepsilon \mu_{i+1})$ such that $z$ can be connected to $p_{i+1}$ using only wormhole levels of order $n\geq B$. Then we obtain
\begin{equation}\label{J5b}    
\mu(S\cap B(p_{i+1},\varepsilon \mu_{i+1}))\geq \frac{\min(w,1-w)}{4}\mu(B(p_{i+1},\varepsilon \mu_{i+1})).
\end{equation}
Indeed, the ball $B(p_{i+1},\varepsilon \mu_{i+1})$ intersects at most three sets $K_{a_j} \times I$, where $a_j$ are finite strings of length $B-1$, $j=1,2,3$. Without loss of generality, assume $p_{i+1} \in  I\times K_{a_1}$ and note that each $a_j$ differs from $a_1$ in at most one entry. Define the sets $Q_j =  B(p_{i+1},\varepsilon \mu_{i+1})\cap (I\times K_{a_j})$. Note $S\cap B(p_{i+1},\varepsilon \mu_{i+1})= Q_{1}$. For each $j$ we have, as a consequence of Fubini's theorem,
\[ \mu(Q_{j})\leq \frac{\mu(Q_{1})}{\min(w,1-w)}.\]
Hence
\[B(p_{i+1},\varepsilon \mu_{i+1})=\sum_{j=1}^{3}\mu(Q_{j})\leq \frac{3\mu(Q_{1})}{\min(w,1-w)}.\]
Hence
\[\mu(Q_{1})\geq \frac{\min(w,1-w)}{3}\mu(B(p_{i+1},\varepsilon \mu_{i+1}))\]
which gives the desired inequality.

Using \eqref{J5a} and \eqref{J5b} we can choose a point $q_{i+1}$ with
		\[q_{i+1} \in S\cap B(p_{i+1},\varepsilon \mu_{i+1})\cap D_{\varepsilon}(x) \cap E_{K}^{1}(\varepsilon).\]
		Clearly by definition $q_{i+1}$ satisfies (1) and (3). 
		
		Next, define $q_{i}$ from $q_{i+1}$ as stated in (4). Then $q_{i}$ can be reached from $p_{i}$ from a vertical displacement at most $2 \varepsilon \mu_{i+1}$ and wormhole levels of order $n\geq B$. Such jump levels are spaced by at most $2/3^{B}$ in the vertical direction. Hence
		\[d(q_{i},p_{i})\leq 2 \varepsilon \mu_{i+1}+(4/3^B)\leq 6 \varepsilon \mu_{i+1}.\]
		This shows that $q_{i}$ satisfies (2) and completes the proof.
	\end{proof}

 Using Lemma \ref{qs} and the same steps as in \cite{CPS22} yields the estimate 	
 \begin{align*}
		&|f(p_{i})-f(p_{i+1})-f_{I}(x)(h(p_{i})-h(p_{i+1}))| \label{lineest}\\
		&\qquad \leq (2L+2)\varepsilon |h(p_{i})-h(p_{i+1})|+7L\varepsilon\mu_{i+1} \nonumber.
	\end{align*}
 Adding these estimates over all $i\geq 0$ gives
 	\begin{align*}
		&|f(y)-f(x)-f_{I}(x)(h(y)-h(x))|\\
		&\qquad \leq (2L+2)\varepsilon d(x,y) + 77L\varepsilon d(x,y).
	\end{align*}
Using a similar argument to that of \cite{CPS22} then concludes the proof of Theorem \ref{rademachermuw}.

\section{Poincar\'e inequality}\label{sec:pi}
In this section, we give an argument that $(F,d,\mu)$ satisfies also a \emph{$(1,1)$-Poincar\'e inequality}: This means there exist constants $C>0, \lambda \geq 1$ so that
\[
\frac{1}{\mu(B(x,r))}\int_{B(x,r)} |f-f_B| d\mu \leq C r \frac{1}{\mu(B(x,\lambda r))} \int_{B(x,\lambda r)} \Lip[f] d\mu,
\]
holds for all Lipschitz functions $f:F\to \R$ and all balls $B=B(x,r)\subset F$. Here, $\Lip[f](x)=\limsup_{y\to x} \frac{|f(y)-f(x)|}{d(x,y)}$, and $f_A=\frac{1}{\mu(A)} \int_A f d\mu$ for Borel sets $A\subset X$ with $\mu(A)>0$. See \cite{Che99} and \cite{shabook} for further background on the Poincar\'e inequality, and \cite[Theorem 2]{keith} for relationships between equivalent formulations of the Poincar\'e inequality. In \cite{Che99}, it was shown that a doubling metric measure space satisfying a Poincar\'e inequality satisfies a notion of differentiability with respect to a collection of charts. Theorem \ref{Rademacher} shows that the chart constructed in \cite{Che99} can be chosen as $(F,h)$. Thus, by proving the Poincar\'e inequality, we establish a closer connection between this work and \cite{Che99}. Further, we show that the present examples are similar to the ones by Schioppa in \cite{Sch15}.

The argument is based on using a pointwise version of the Poincar\'e inequality: There exist constants $C,\lambda \geq 1$ so that for all Lipschitz functions $f:F\to \R$ and all points $p,q\in X$ we have:
\[
|f(p)-f(q)|\leq Cd(x,y)( M_{\lambda d(x,y)} \Lip[f](p) + M_{\lambda d(x,y)} \Lip[f](q)),
\]
where 
\[
M_R h (x) = \sup_{r\in (0,R)} \frac{1}{\mu(B(x,r))} \int_{B(x,r)} |h| d\mu
\]
is the Hardy-Littlewood maximal function. By a result from \cite[Theorem 8.1.7]{shabook} a $(1,1)$-Poincar\'e inequality is equivalent to a pointwise $(1,1)$-Poincar\'e inequality. This result is originally due to Haj{\l}asz and Koskela, see e.g. \cite{hajlaszkoskela}.  We will prove the pointwise version of the Poincar\'e inequality by a ``chaining of balls''-type argument, although in our case, we will chain rectangles of the form $J\times K_x$. This type of argument is also due to Haj{\l}asz and Koskela.

Our argument will use the following simple one-dimensional result. Let $J\subset \R$ be an interval, 
$f:J\to \R$ a Lipschitz function, and let $A,B\subset J$ be positive measure subsets. Then we have the following inequality:
\begin{equation}\label{eq:onedpoincare}
\left|\frac{1}{|A|} \int_A f(t) dt - \frac{1}{|B|} \int_B f(t) dt \right|\leq \int_J \Lip[f](t) dt,
\end{equation}
where $|A|$ is the Lebesgue measure of the set $A$.
Indeed, for every $a\in A, b\in B$, we have $|f(a)-f(b)|\leq \int_J \Lip[f](t) dt$. Taking an average integral in both $a\in A$ and $b\in B$ yields \eqref{eq:onedpoincare}. 

To simplify the presentation of the proof below, we will use $\mathcal{A}\lesssim \mathcal{B}$ to indicate that there is a constant $C$ so that $\mathcal{A}\leq C\mathcal{B}$. The constant $C$ in all instances will only depend on the space $F$ and not on the function $f$, or other variables in the proof.

\begin{proof} In the proof, we denote by $\mu$ the measure $\mu_w$ and by $\nu$ the measure $\nu_w$. First, we prove three inequalities, \eqref{eq:simplePI1}, \eqref{eq:simplePI2}, \eqref{eq:simplePI3}, where one can control the differences of averages over sets of the form $q(J\times K_x)$, called rectangles, which lie ``near'' to each other in specific ways. The three cases are showed in Figure \ref{fig:cases}.

First, Case A), where two rectangles are connected through a wormhole: Let $q(J\times K_x)$ and $q(J\times K_{x'})$ be sets with $|x|=|x'|$, $J\subset [0,1]$ is a sub-interval and for which the strings $x$ and $x'$ differ only at the $n'$th bit for some $1\leq n\leq |x|$, and $t_n \in J$ for some  wormhole level $t_n\times K$ of order $n$. Let $s$ be a point in $K_x$, which via its trinary expansion can be identified by an infinite trinary string, and let $s'\in K_{x'}$ be the infinite trinary string obtained from $s$ with the $n'$th bit  flipped. Equivalently $s'$ is obtained from $s$ via the identification in Definition \ref{def:wormhole}, i.e. a translation to the right or left by $2(3^{-n})$, depending on if the n'th bit of $s$ is a 0 or 2, respectively. Then, for $t\in [0,3|J|]$ we define
\[
\gamma_{s}(t)=\begin{cases} [(\max(J)-t,s)] & t\in [0,|J|] \\
[(\min(J)+(t-|J|),s)] & t\in [|J|,|J|+t_n-\min(J)] \\ 
[(\min(J)+(t-|J|),s')] & t\in [|J|+t_n-\min(J),2|J|] \\ 
[(\max(J)-(t-2|J|),s')] & t\in [2|J|,3|J|] \\
\end{cases}
\]

Notice $\gamma_s(|J|)=[(\min(J),s)]$ from both the first and second line, $\gamma_s(|J|+t_n-\min(J))=[(t_n,s)]=[(t_n,s')]$ from the second and third line and $\gamma_s(2|J|)=[(\max(J),s')]$ from the final two lines. See Figure \ref{fig:gamomega} for a figure of the curve $\gamma_s(t)$.

\begin{figure}[!ht]
\begin{tikzpicture}[scale=0.3]
\draw (5.75,19.75) -- (5.75,9.5);
\draw  (10,9.5) -- (10,19.75);
\draw [dashed] (5.75,17.25) -- (10,17.25);
\draw [->,thick] (5.25,19.75) -- (5.25,9.5);
\draw [->,thick] (6.25,9.5) -- (6.25,17.25);
\draw [->,thick] (9.5,17.25) -- (9.5,19.75);
\draw [->,thick] (10.5,19.75) -- (10.5,9.5);
\node  at (5.75,8.75) {$J\times s$};
\node  at (10,8.75) {$J\times s'$};
\end{tikzpicture}
\caption{The path $\gamma_s$. The dashed line shows where the wormhole is used.}
\label{fig:gamomega}
\end{figure}
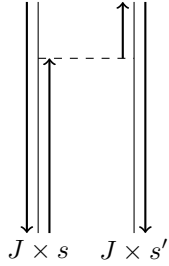

The curve $\gamma_s$ is a unit speed curve, and $f\circ\gamma_{s}(t)$ is Lipschitz. First, by using the fact that $\mu = q_*(\calH^1 \times \nu)$, we get by a change of variables and the fact $\mu(q(J\times K_x))=|J|\mu(K_x)$ that
\begin{align*}
\frac{1}{\mu(q(J\times K_x))}\int_{q(J\times K_x)} f d\mu&= \frac{1}{|J|\nu(K_x)}\int_{J\times K_x} f(q(t,s)) dt d\nu(s) \\
&= \frac{1}{|J| \nu(K_x)}\int_{K_x} \int_0^{|J|} f(\gamma_{s}(t))dt d\nu(s).
\end{align*}
Next, we use fact that the map $T(s)=s', T:K_x\to K_{x'}$ is a translation, and the push-forward is given by $T_*(\nu|_{K_x})=\frac{\nu(K_x)}{\nu(K_{x'})} \nu|_{K_{x'}}$. This follows from the definition of $\nu$ and Proposition \ref{prop:measure}, since if $a$ is any finite string then $T_*\nu(K_{x'a})=\nu(K_{xa})=\nu(K_x)\nu(K_a)$ and $\nu(K_{x'a})=\nu(K_{x'})\nu(K_a)$. Since the sets $K_{xa}$ generate the sigma-algebra, this shows that the measures $T_*(\nu|_{K_x})$ and $\nu|_{K_{x'}}$ differ only by the factor $\nu(K_{x})/\nu(K_{x'})$. This, together with the same change of variables, yields 
\begin{align*}
\frac{1}{\mu(q(J\times K_{x'}))}\int_{q(J\times K_{x'})} f d\mu&= \frac{1}{|J|\nu(K_{x'})}\int_{J\times K_{x'}} f(q(t,s')) dtd\nu(s') \\
&=  \frac{1}{|J| \nu(K_{x})}\int_{K_x} \int_{2|J|}^{3|J|} f(\gamma_{s}(t)) dt d\nu(s).
\end{align*}

Thus, \eqref{eq:onedpoincare} and integration over $s$ gives
\begin{align}
\label{eq:simplePI1}
&\left|\frac{1}{\mu(q(J\times K_x))}\int_{q(J\times K_x)} fd\mu - \frac{1}{\mu(q(J\times K_{x'}))}\int_{q(J\times K_{x'})} f d\mu\right|  \\
& \leq\frac{1}{|J| \nu(K_{x})}\int_{K_x} \left|\int_{0}^{|J|} f(\gamma_{s}(t)) dt  - \int_{2|J|}^{3|J|} f(\gamma_{s}(t)) dt\right| d\nu(s) \nonumber \\
& \leq  2|J| \frac{1}{\mu(q(J\times K_{x'}))}\int_{q(J\times K_{x'})} \Lip[f](t) d\mu + 2|J|\frac{1}{\mu(q(J\times K_{x}))}\int_{q(J\times K_{x})} \Lip[f](t) d\mu.\nonumber
\end{align}

Next, we consider Case $B)$, where the two rectangles are of the form $q(J\times K_x)$ and $q(J'\times K_x)$ with $J'\subset J$ and $J',J$ are subintervals of $[0,1]$. Then \eqref{eq:onedpoincare} implies
\begin{align}
\left|\frac{1}{\mu(q(J\times K_{x}))}\int_{J\times K_{x}} f d\mu \right. &-\left.  \frac{1}{\mu(q(J'\times K_{x}))}\int_{q(J'\times K_{x})} f d\mu 
\right| \nonumber \\
&\leq |J| \frac{1}{\mu(q(J\times K_{x}))}\int_{q(J\times K_{x})} \Lip[f] d\mu. \label{eq:simplePI2}
\end{align}
Finally, consider the Case C) consisting of rectangles $q(J\times K_x)$ and $q(J'\times K_{x'})$, where $J'\subset J$,  $x'$ is obtained from the string $x$ by adding one bit and $J$ contains a wormhole at level $|x|+1$. By symmetry consider only the case, where $w'=w0$ and consider three averages $Q_1=\frac{1}{\mu(q(J'\times K_{x0}))}\int_{q(J'\times K_{x0})} fd\mu$, $Q_2=\frac{1}{\mu(q(J\times K_{x0}))}\int_{q(J\times K_{x0})} fd\mu$ and $Q_3=\frac{1}{\mu(q(J\times K_{x1}))}\int_{q(J\times K_{x1})} fd\mu$. Then
\begin{align}
\nonumber
\left|\frac{1}{\mu(q(J\times K_{x}))}\int_{q(J\times K_{x})} fd\mu \right. &-\left. \frac{1}{\mu(q(J'\times K_{x'}))}\int_{q(J'\times K_{x'})} fd\mu 
\right| \\
&= 
\left| (wQ_2+(1-w)Q_3)-Q_1\right|  \nonumber \\
&\leq (1-w) |Q_2-Q_3| + |Q_2-Q_1| \nonumber \\
&\lesssim |J| \frac{1}{\mu(q(J\times K_{x}))}\int_{q(J\times K_{x})} \Lip[f] d\mu. \label{eq:simplePI3}
\end{align}
In the last line, we used \eqref{eq:simplePI2} to estimate the difference $|Q_2-Q_1|$ and \eqref{eq:simplePI1} to estimate the difference $|Q_2-Q_3|$.

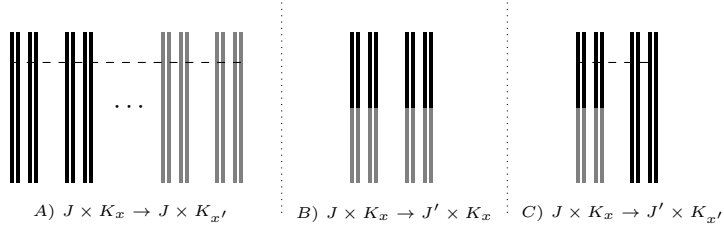
\begin{figure}[!ht]
\begin{tikzpicture}[scale=0.4]
\fill [fill=black, draw=none] (0,0) rectangle (.1,5);
\fill [fill=black, draw=none](.2,0) rectangle (.3,5);
\fill [fill=black, draw=none] (0.6,0) rectangle (.7,5);
\fill [fill=black, draw=none](.8,0) rectangle (.9,5);

\fill [fill=black, draw=none] (1.8,0) rectangle (1.9,5);
\fill [fill=black, draw=none](2.0,0) rectangle (2.1,5);
\fill [fill=black, draw=none] (2.4,0) rectangle (2.5,5);
\fill [fill=black, draw=none](2.6,0) rectangle (2.7,5);

\node at (4,2.5) {$\cdots$};
\node at (4,-1) { \tiny $A)\  J\times K_x \to J\times K_{x'}$};

\fill [fill=gray, draw=none] (5,0) rectangle (5.1,5);
\fill [fill=gray, draw=none](5.2,0) rectangle (5.3,5);
\fill [fill=gray, draw=none] (5.6,0) rectangle (5.7,5);
\fill [fill=gray, draw=none](5.8,0) rectangle (5.9,5);

\fill [fill=gray, draw=none] (6.8,0) rectangle (6.9,5);
\fill [fill=gray, draw=none](7.0,0) rectangle (7.1,5);
\fill [fill=gray, draw=none] (7.4,0) rectangle (7.5,5);
\fill [fill=gray, draw=none](7.6,0) rectangle (7.7,5);

\draw [dashed] (0,4) -- (7.7,4);

\draw [dotted] (9,-1) -- (9,6);

\fill [fill=gray, draw=none] (11.3,0) rectangle (11.4,2.5);
\fill [fill=gray, draw=none](11.5,0) rectangle (11.6,2.5);
\fill [fill=gray, draw=none] (11.9,0) rectangle (12.0,2.5);
\fill [fill=gray, draw=none](12.1,0) rectangle (12.2,2.5);

\fill [fill=black, draw=none] (11.3,2.5) rectangle (11.4,5);
\fill [fill=black, draw=none](11.5,2.5) rectangle (11.6,5);
\fill [fill=black, draw=none] (11.9,2.5) rectangle (12.0,5);
\fill [fill=black, draw=none](12.1,2.5) rectangle (12.2,5);

\node at (12.75,-1) {\tiny $B)\  J\times K_x \to J'\times K_{x}$};

\fill [fill=gray, draw=none] (13.1,0) rectangle (13.2,2.5);
\fill [fill=gray, draw=none](13.3,0) rectangle (13.4,2.5);
\fill [fill=gray, draw=none] (13.7,0) rectangle (13.8,2.5);
\fill [fill=gray, draw=none](13.9,0) rectangle (14,2.5);

\fill [fill=black, draw=none] (13.1,2.5) rectangle (13.2,5);
\fill [fill=black, draw=none](13.3,2.5) rectangle (13.4,5);
\fill [fill=black, draw=none] (13.7,2.5) rectangle (13.8,5);
\fill [fill=black, draw=none](13.9,2.5) rectangle (14,5);

\draw [dotted] (16.5,-1) -- (16.5,6);

\fill [fill=gray, draw=none] (18.8,0) rectangle (18.9,2.5);
\fill [fill=gray, draw=none](19.0,0) rectangle (19.1,2.5);
\fill [fill=gray, draw=none] (19.4,0) rectangle (19.5,2.5);
\fill [fill=gray, draw=none](19.6,0) rectangle (19.7,2.5);

\fill [fill=black, draw=none] (18.8,2.5) rectangle (18.9,5);
\fill [fill=black, draw=none](19.0,2.5) rectangle (19.1,5);
\fill [fill=black, draw=none] (19.4,2.5) rectangle (19.5,5);
\fill [fill=black, draw=none](19.6,2.5) rectangle (19.7,5);

\node at (20.35,-1) {\tiny $C)\  J\times K_x \to J'\times K_{x'}$};

\fill [fill=black, draw=none] (20.6,0) rectangle (20.7,5);
\fill [fill=black, draw=none](20.8,0) rectangle (20.9,5);
\fill [fill=black, draw=none] (21.2,0) rectangle (21.3,5);
\fill [fill=black, draw=none](21.4,0) rectangle (21.5,5);

\draw [dashed] (18.8,4) -- (21.5,4);

\end{tikzpicture}
\caption{The three cases for rectangles that we consider. One of the sets is shaded gray and the other set is shaded black. In the middle case, the gray shading overlaps with the black shading. The dashed lines show wormhole levels. Notice how Case C can be decomposed into parts, which are related via Case A and Case B.}
\label{fig:cases}
\end{figure}

With these simple estimates given, we construct a chain of sets of the form $q(J_i\times K_{x_i})$, $i\in \Z$, which connect every pair of points, and where consecutive sets are related to each other as in one of the three cases considered above.  Let $p=[x_1,x_2],q=[y_1,y_2]\in F$ be two distinct points. Choose $n\in \N$ so that $2(3^{-n-1})<d(p,q) \leq 2(3^{-n})$. Let $J_0$ be a shortest interval containing $x_1,y_1$ and which contains all the wormhole levels needed to connect $p$ to $q$ and which has length comparable to $d(p,q)$, with a constant independent of $p$ and $q$. Let $m\leq 0$ and define $K_{m}^1$ to be the $n+m-1$'th level interval in the Cantor set containing $x_2$, and for $m>0$ let $K_m^2$ to be the $n+m-2$'th level interval in the Cantor set containing $y_2$. First, let $J_1^1=J_0=J_1^2$. Further, let $
\{J_m^1\}_{m=2}^\infty$ be a nested sequence of intervals of lengths $3^{-n-m}$ which contains $x_1$ and where each interval is contained in $J_0$. Symmetrically, let $
\{J_m^2\}_{m=2}^\infty$ be a nested sequence of intervals of lengths $3^{-n-m}$ which contains $y_1$ and where each interval is contained in $J_0$.

Now, consider the sets $Q_i$ defined as follows. $Q_0=q(J_0\times K_0^1)$, and $Q_m = q(J_m^1\times K_{|m|}^1)$ for $m< 0$ and $Q_m = q(J^2_m\times K_{|m|}^2)$ for $m> 0$.  By construction we can write $Q_m= q(J_m\times K_{x_m})$ with $|J_m|\leq \diam(Q_m)\lesssim d(p,q) 3^{-|m|}$. Notice that for all $m\in \Z$, the sets $Q_m$ and $Q_{m+1}$ relate to each other as one of the three cases considered in the beginning of the proof. First, for each $i<0$, we have that $Q_i$ is obtained from $Q_{i+1}$ as in case C), since $J_i^1 \subset J_{i+1}^1$ and $K_{i-1}^1$ is the left or right half of $K_{i}^1$. Similarly for $i>0$, we have that $Q_i$ and $Q_{i+1}$ are related as in C). We are left to consider $Q_0$ and $Q_1$. We have $J_1^2=J_0$ and $K_0^1=K_a$ and $K_1^2=K_b$ for some finite strings $a,b$ with $|a|=|b|$. By Lemma \ref{atmostone}, the strings $a,b$ can differ from each other by at most one bit of index $\leq n-1$, and there is a wormhole level $t\in J_0$ corresponding to this level. Thus, we can apply Case A) to estimate the difference between $f_{Q_0}$ and $f_{Q_1}$.

Now, by combining continuity of $f$ with a telescoping sum argument as well as \eqref{eq:simplePI3}, \eqref{eq:simplePI2} and \eqref{eq:simplePI1}, we get
\begin{align}
|f(p)-f(q)|&\leq \sum_{i\in \Z}\left|\frac{1}{\mu(Q_i)}\int_{Q_i} f d\mu - \frac{1}{\mu(Q_{i+1})}\int_{Q_{i+1}} f d\mu\right| \nonumber \\
&\lesssim \sum_{i\in \Z}\diam(Q_i)\frac{1}{\mu(Q_i)}\int_{Q_i} \Lip[f] d\mu +  \diam(Q_{i+1})\frac{1}{\mu(Q_{i+1})}\int_{Q_{i+1}} \Lip[f] d\mu \nonumber\\
&\lesssim \sum_{i\in \Z}\diam(Q_i)\frac{1}{\mu(Q_i)}\int_{Q_i} \Lip[f] d\mu.
\label{eq:sumpi}
\end{align}
In the last line, we observed that the sum over $i$ of the two terms on the second line are equal, since they are obtained by shifting indices $i\to i+1$. Notice that for $i\leq 0$ we have $d(Q_i, p)\lesssim 3^{-|i|}d(p,q)$, and that $\mu(Q_i) \gtrsim \mu(B(p,3^{-|i|}d(p,q)))$ by doubling since $Q_i$ contains a ball with radius comparable to $3^{-|i|}d(p,q)$. Thus, we see that there is a constant $\lambda >1$ independent of $p$ and $q$ for which  
\[
\frac{1}{\mu(Q_i)}\int_{Q_i} \Lip[f] d\mu \lesssim M_{\lambda d(p,q)} \Lip[f](p) \text{ for $i\leq 0$.}
\]
Similarly, 
\[
\frac{1}{\mu(Q_i)}\int_{Q_i} \Lip[f] d\mu \lesssim M_{\lambda d(p,q)} \Lip[f](q) \text{ for $i>0$.}
\]
Thus, from this, \eqref{eq:sumpi} and $\diam(Q_i)\lesssim 3^{-|i|}d(p,q)$ we get some constant $C>1$ independent of the points $p,q$ and the function $f$ for which 
\[
|f(p)-f(q)|\leq Cd(p,q) (M_{\lambda d(p,q)}\Lip[f](p)+M_{\lambda d(p,q)}\Lip[f](q)).
\]
This is the pointwise Poincar\'e inequality, and by \cite[Theorem 8.1.7]{shabook} this implies the $(1,1)$-Poincar\'e inequality.
\end{proof}


\begin{thebibliography}{50}
\bibitem{ACP10} Alberti, G., sCsornyei, M., Preiss, D.: \emph{Differentiability of Lipschitz Functions, Structure of Null Sets, and Other Problems}, Proc. Int. Congress Math. III, (2010), 1379--1394.
\bibitem{Bil76} Billingsley, P.: \emph{Probability and Measure}, 3rd Edition, John Wiley \& Sons, Inc., 1995, ISBN 0-471-00710-2.
\bibitem{BL00}  Benyamini, Y., Lindenstrauss, J.: \emph{Geometric nonlinear functional analysis. Vol. 1.}, American Mathematical Society Colloquium Publications, 48. American Mathematical Society, Providence, RI, 2000. xii+488 pp. ISBN: 0-8218-0835-4.
\bibitem{Buc99} Buckley, S.: \emph{Is the Maximal Function of a Lipschitz Function Continuous?} Annales Academiae Scientiarum Fennicae Mathematica 24 (1999), 519--528.
\bibitem{CPS22} Capolli, M., Pinamonti, A., Speight, G.: \emph{Maximal Directional Derivatives in Laakso Space}, To appear in Comm. Contemp. Math. https://doi.org/10.1142/S0219199724500172
\bibitem{Cap} Capolli, M.: \emph{An Overview on Laakso Spaces}, To appear in Note di Matematica. Preprint available at arXiv:2211.05681
\bibitem{Che99} Cheeger, J.: \emph{Differentiability of Lipschitz Functions on Metric Measure Spaces}, Geometric and Functional Analysis 9(3) (1999), 428--517.
\bibitem{CJ15} Csornyei, M., Jones, P.: \emph{Product formulas for measures and applications to analysis and geometry}, announcement of result is in slides available at: http://www.math.sunysb.edu/Videos/dfest/PDFs/38-Jones.pdf.
\bibitem{CK15} Cheeger, J., Kleiner, B.: \emph{Inverse Limit Spaces Satisfying a Poincar\'e Inequality}, Analysis and Geometry in Metric Spaces 2015; 1:15--39.
\bibitem{DM11} Dor\'e, M., Maleva, O.: \emph{A compact null set containing a differentiability point of every Lipschitz function}, Mathematische Annalen 351(3) (2011), 633--663.
\bibitem{DM14} Dymond, M., Maleva, O.: \emph{Differentiability inside sets with upper Minkowski dimension one}, Michigan Mathematical Journal 65 (2016), 613--636.
\bibitem{DMMPR} De Philippis, G., Marchese, A., Merlo, A., Pinamonti, A., Rindler, F.: \emph{On the converse of Pansu's Theorem},  Arch. Rational Mech. Anal. 249, 3 (2025). https://doi.org/10.1007/s00205-024-02059-8
\bibitem{DR16} De Philippis, G., Rindler, F.: \emph{On the structure of A-free measures and applications}, Annals of Mathematics 184 (2016), 1017–1039.
\bibitem{SEB} Eriksson-Bique, S.: \emph{Characterizing spaces satisfying Poincar\'{e} inequalities and
   applications to differentiability, Geom. Funct. Anal.} 29 (1) (2019), 119--189.
\bibitem{Fit84} Fitzpatrick, S.: \emph{Differentiation of Real-Valued Functions and Continuity of Metric Projections}, Proceedings of the American Mathematical Society 91(4) (1984), 544--548.
\bibitem{hajlaszkoskela} Haj{\l}asz, P., Koskela, P., \emph{Sobolev met Poincaré}, Mem. Am. Math. Soc. 688, 101 p. (2000).
\bibitem{keith} Stephen, K. \emph{Modulus and the Poincaré inequality on metric measure spaces}, Mathematische Zeitschrift 245(2): 255-292, (2003).
\bibitem{Laa00} Laakso, T. J., \emph{Ahlfors $Q$-Regular Spaces with Arbitrary $Q>1$ Admitting Weak Poincar\'e Inequality}, Geometric and Functional Analysis 10(1), 111--123, 2000.
\bibitem{LPS17} Le Donne, E., Pinamonti, A., Speight, G.: \emph{Universal Differentiability Sets and Maximal Directional Derivatives in Carnot Groups}, Journal de Math\'{e}matiques Pures et Appliqu\'{e}es 121 (2019), 83--112.
\bibitem{Pan89} Pansu, P.: \emph{Metriques de Carnot-Caratheodory et quasiisometries des espaces symwtriques de rang un}, Annals of Mathematics 129(1) (1989), 1-60.
\bibitem{PS15} Preiss, D., Speight, G.: \emph{Differentiability of Lipschitz Functions in Lebesgue Null Sets}, Inventiones Mathematicae 199(2) (2015), 517--559.	
\bibitem{PS16} Pinamonti, A., Speight, G.: \emph{A Measure Zero Universal Differentiability Set in the Heisenberg Group}, Mathematische Annalen 368, no. 1-2, (2017) 233--278.
\bibitem{PS18} Pinamonti, A., Speight, G.: \emph{Universal Differentiability Sets in Carnot Groups of Arbitrarily High Step}, Israel Journal of Mathematics, (2019).
\bibitem{Pre90} Preiss, D.: \emph{Differentiability of Lipschitz Functions on Banach spaces}, Journal of Functional Analysis 91(2) (1990), 312--345.
\bibitem{shabook} Heinonen, J., Koskela, P., Shanmugalingam, N., and Tyson, J.~T., {\em Sobolev spaces on metric measure spaces}, vol.~27 of {\em New Mathematical Monographs}. Cambridge University Press, Cambridge, 2015.
\bibitem{Sch15} Schioppa, A.: \emph{Poincar\'e Inequalities for Mutually Singular Measures}, Analysis and Geometry in Metric Spaces 2015; 3:40-45.
\bibitem{Zah46} Zahorski, Z.: \emph{Sur l'ensemble des points de non-d\'erivabilit\'e d'une fonction continue}, Bulletin de la Soci\'et\'e Math\'ematique de France 74 (1946), 147--178.
\end{thebibliography}
\end{document}